\documentclass[12pt]{article}

\usepackage{amsmath,amssymb,amsfonts, amsthm}
\usepackage{amscd}
\usepackage{bbm,bm} 
\usepackage{cases}
\usepackage[mathscr]{euscript}

\usepackage{comment}

\usepackage{breakcites}
\usepackage{natbib}
\RequirePackage[colorlinks,citecolor={blue!60!black},urlcolor={blue!70!black},linkcolor={red!60!black},breaklinks,hypertexnames=false]{hyperref}

\usepackage[ruled,vlined]{algorithm2e}
\usepackage{array}
\usepackage[small]{caption}
\usepackage{epic,eepic, epsfig, psfrag}
\usepackage{enumerate}
\usepackage{fullpage}
\usepackage{graphicx}
\usepackage[sf,bf,SF,footnotesize]{subfigure}
\usepackage{setspace}
\usepackage{tikz}
\usepackage{xcolor}

\usepackage{etoolbox}

\newtheorem{theorem}{Theorem}

\newtheorem{lemma}{Lemma}

\newtheorem{corollary}{Corollary}

\newcommand{\N}{\mathbb{N}}
\newcommand{\R}{\mathbb{R}}

\newcommand{\E}{\mathbb{E}}
\newcommand{\one}{\mathbbm{1}}
\newcommand{\X}{\mathcal{X}}
\newcommand{\Prob}{\mathbb{P}}

\newcommand{\kullbackLeibler}{\mathrm{kl}}

\newcommand{\massartKLLowerBoundFunction}{h_m}

\usepackage[normalem]{ulem}

\title{A short proof of the Dvoretzky--Kiefer--Wolfowitz--Massart inequality}
\author{Henry W J Reeve\\University of Bristol\\henry.reeve@bristol.ac.uk}

\date{}

\newcommand{\event}{\mathcal{E}}

\newcommand{\filtrationSigmaAlgebra}{\mathcal{G}}

\begin{document}

\maketitle

\newcommand{\topologicalClosure}{\mathrm{cl}}

\newcommand{\populationCDF}{F}
\newcommand{\empiricalCDF}{\hat{F}_n}
\newcommand{\populationDistribution}{\nu}
\newcommand{\empiricalDistribution}{\hat{\nu}_n}
\newcommand{\modulusOfContinuityKL}{\omega}

\newcommand{\ucbByLambda}{\hat{U}_{n,\lambda,\delta}}

\newcommand{\bennettFunction}{h}
\newcommand{\bernsteinFunction}{h_1}

\newcommand{\limitLambdaSmallFunction}{g}
\newcommand{\errorFunctionZero}{\Delta}
\newcommand{\errorFunctionOne}{\Delta_+}

\newcommand{\errorFunctionZeroRestriction}{\errorFunctionZero}
\newcommand{\twoVariableFunctionDerivPFunction}{\twoVariableFunction_{0}}
\newcommand{\twoVariableFunctionDerivLambdaFunction}{\twoVariableFunction_{1}}

\newcommand{\lambdaFunctionOfP}{\underline{\lambda}}
\newcommand{\pFunctionOfLambda}{\overline{p}}
\newcommand{\pFunctionOfLambdaDifferentiable}{\tilde{r}}

\newcommand{\fixedPointP}{p_\star}
\newcommand{\fixedPointLambda}{\lambda_\star}
\newcommand{\fixedPointQ}{q_\star}
\newcommand{\lambertW}{W_0}
\newcommand{\lambdaMinByEpsilon}{\lambda_{\min}(\varepsilon)}

\newcommand{\pMin}{p_{\min}}
\newcommand{\pMax}{p_{\max}}
\newcommand{\probabilityInterval}{\mathcal{P}}
\newcommand{\intervalReals}{\mathcal{I}}

\newcommand{\epsByNDelta}{\varepsilon(n,\delta)}

\newcommand{\normalisedEmpiricalCDFProcess}{\hat{G}_n}

\sloppy
\begin{abstract} The Dvoretzky--Kiefer--Wolfowitz--Massart inequality gives a sub-Gaussian tail bound on the supremum norm distance between the empirical distribution function of a random sample and its population counterpart. We provide a short proof of a result that improves the existing bound in two respects. First, our one-sided bound holds without any restrictions on the failure probability, thereby verifying a conjecture of \cite{birnbaum1958distribution}. Second, it is local in the sense that it holds uniformly over sub-intervals of the real line with an error rate that adapts to the behaviour of the population distribution function on the interval. \footnote{\thanks{I would like to express my gratitude to Richard J. Samworth for encouragement and useful feedback which substantially improved the quality of the exposition.}}
\end{abstract}

Let $X_1,\ldots,X_n$ be independent random variables with distribution function $\populationCDF:\R \rightarrow [0,1]$. Let $\empiricalCDF$ denote the empirical distribution function,
\begin{align*}
\empiricalCDF(t):=\frac{1}{n}\sum_{i=1}^n \one_{\{ X_i\leq t\}},
\end{align*}
for $t \in \R$. Understanding the deviation of the empirical distribution function from its population counterpart is a longstanding goal within probability theory with a multitude of applications to statistics. The Glivenko--Cantelli theorem \citep{glivenko1933sulla,cantelli1933sulla} gives almost sure convergence of the empirical distribution function to the population distribution function with respect to the supremum norm, initiating a line of work towards understanding the rate of convergence. Classical results of \cite{bernstein1924modification}, and subsequent work by \cite{bennett1962probability} and \cite{hoeffding1963}, may be viewed as giving finite-sample pointwise guarantees for specific value of $t \in \R$. \cite{dvoretzky1956asymptotic} proved a sub-Gaussian bound on the supremum norm distance between the empirical and population distribution functions. Subsequent work of \citep{devroye1979recovery,hu1985uniform,shorack1986empirical} led to improvements in the bound. Motivated by an asymptotic analysis due to \cite{smirnov1944approximate}, and further numerical computations,  \cite{birnbaum1958distribution} gave a conjecture on the optimal leading constant. A breakthrough result of \cite{massart1990tight} achieved the optimal constant, settling the conjecture of \cite{birnbaum1958distribution}, subject to a mild constraint on the failure probability. It follows from \cite{smirnov1944approximate} that the leading term in \cite{massart1990tight} cannot be improved. We recommend \cite{dudley2014uniform} for an exposition of the proof of \cite{massart1990tight}, and its connections to \cite{bretagnolle1989hungarian}. \cite{maillard2021local} derived an integral expression for the probability that the empirical distribution function exceeds a given distance from their population counterpart, uniformly over a given interval. Recently, 
\cite{bartl2023variance} and \cite{blanchard2023tight} give more interpretable local bounds.


Let's first introduce some notation. Given $a,b \in [0,1]$ we let \begin{align*}
\kullbackLeibler(a||b):=a\log\bigg(\frac{a}{b}\bigg)+(1-a)\log\bigg(\frac{1-a}{1-b}\bigg),
\end{align*}
where we adopt the convention that $0\log 0=0\log (0/0)=0$ and $z\log (z/0)=\infty$ for $z>0$. Thus, $\kullbackLeibler(a||b)$ denotes the Kullback-Leibler divergence between two Bernoulli random variables with respective success probabilities $a$ and $b$. For each $p \in (0,e^{-\varepsilon}]$ there is a unique value of $\eta \in [0,1-p]$ with $\kullbackLeibler(p+\eta ||p)=\varepsilon$ (Lemma \ref{lemma:modulusContinuityWellDefinedByKLProperty}) and we denote this unique value by $\modulusOfContinuityKL(p,\varepsilon)$. We extend $\modulusOfContinuityKL$ to $[0,1]\times [0,\infty)$ by letting $\modulusOfContinuityKL(p,\varepsilon):=1-p$ if $p \in (e^{-\varepsilon},1]$; we also let $\modulusOfContinuityKL(p,\varepsilon):=0$ whenever $p=0$, so that $p \mapsto \modulusOfContinuityKL(p,\varepsilon)$ is continuous (Lemma \ref{lemma:checkContinuityModulusIsCont}). Given $\delta \in (0,1)$ we let $\epsByNDelta:=\log(1/\delta)/n$.

Our main result is the following.

\begin{theorem}\label{thm:main_result_localised_dkw_m_inequality} Given any interval $\intervalReals \subseteq  \R$ and $\delta \in (0,1)$ we have
\begin{align*}
\Prob\biggl\{ \sup_{t  \in \intervalReals}\bigl\{  \empiricalCDF(t)-\populationCDF(t)\bigr\} > \sup_{t \in \intervalReals}\,\modulusOfContinuityKL\bigl(\populationCDF(t),\epsByNDelta\bigr)\biggr\} \leq \delta.
\end{align*} 
\end{theorem} 

\newcommand{\interpretableLocalBennettBoundSingleProb}{\varphi_0}
\newcommand{\interpretableLocalBernsteinBoundSingleProb}{\varphi_1}

\newcommand{\interpretableHoeffdingTypeBound}{\vartheta}
\newcommand{\interpretableLocalBoundInterval}{\varphi}

\newcommand{\interpretableMassartDudleyTypeBoundPointwise}{\varphi_0}
\newcommand{\interpretableMassartDudleyTypeBoundInterval}{\varphi}

\newcommand{\probByEps}{\rho_\varepsilon}
\newcommand{\sign}{\mathrm{sign}}

In order to give a more interpretable bound we write $\sigma(p):=\sqrt{p(1-p)}$ for  $p \in [0,1]$ and introduce the function $\interpretableMassartDudleyTypeBoundPointwise: [0,1]\times (0,\infty) \rightarrow (0,1]$ defined by
\begin{align}\label{def:interpretableBoundFormulaMassartDudleyType}
\interpretableMassartDudleyTypeBoundPointwise(p,\varepsilon):&=\sqrt{4\sigma^2(p)\probByEps^2+ \sigma^2(\probByEps)}+\sigma(\probByEps)(1-2p)\\
& = \sqrt{4\sigma^2(p)\probByEps^2+ \sigma^2(\probByEps)}+\sigma(\probByEps)\,\sign(1-2p)\sqrt{1-4\sigma^2(p)}, \nonumber
\end{align}
where $\probByEps:=9/(9+2\varepsilon) \in (0,1)$. Note that $\interpretableMassartDudleyTypeBoundPointwise(p,\varepsilon)\leq 1$ for all $p \in [0,1]$. Indeed, the map $p\mapsto \interpretableMassartDudleyTypeBoundPointwise(p,\varepsilon)$ is concave with a unique maximum of $\sup_{p \in [0,1]}\interpretableMassartDudleyTypeBoundPointwise(p,\varepsilon)= 1$ attained when $1-2p=\sigma(\probByEps)/\probByEps$. Moreover, we have $\lim_{\varepsilon \searrow 0} \probByEps = 1$, so  $\lim_{\varepsilon \searrow 0} \sigma(\probByEps) = 0$, and hence $\lim_{\varepsilon \searrow 0}\interpretableMassartDudleyTypeBoundPointwise(p,\varepsilon)=2\sigma(p)$ for all $p \in [0,1]$. Given an interval $\intervalReals \subseteq \R$, let $\pMin:=\inf_{t \in \intervalReals} \populationCDF(t)$ and $\pMax:=\sup_{t \in \intervalReals} \populationCDF(t)$, and set
\begin{align*}
\interpretableMassartDudleyTypeBoundInterval(\intervalReals,\varepsilon):= \begin{cases} \interpretableMassartDudleyTypeBoundPointwise(\pMax,\varepsilon) & \text{ if } 2\pMax < 1-\sigma(\probByEps)/\probByEps\\
1 & \text{ if } 2\pMin \leq  1-\sigma(\probByEps)/\probByEps \leq 2\pMax\\
\interpretableLocalBennettBoundSingleProb(\pMin,\varepsilon) & \text{ if } 2\pMin >1-\sigma(\probByEps)/\probByEps.
\end{cases}
\end{align*}
Our first corollary is designed to highlight the role played by variance in the local bound as well as the mild asymmetry in deviations around one half.

\begin{corollary}\label{corr:interpretable_dkw_m_inequality}  Given any interval $\intervalReals \subseteq  \R$ and $\delta \in (0,1)$ we have
\begin{align*}
\Prob\biggl\{ \sup_{t  \in \intervalReals}\bigl\{   \empiricalCDF(t)-\populationCDF(t)\bigr\} >\interpretableMassartDudleyTypeBoundInterval\bigl(\intervalReals,\epsByNDelta\bigr)\sqrt{\frac{\epsByNDelta}{2}} \biggr\} \leq \delta.
\end{align*}
\end{corollary}

Our next corollary is a refinement of the celebrated Dvoretzky--Kiefer--Wolfowitz--Massart inequality \citep{dvoretzky1956asymptotic,massart1990tight}. 

\begin{corollary}\label{corr:dkw_m_inequality_refined_pinsker_strengthening}  Given any $\xi\geq 0$ and $\zeta:=2\xi^2/3$ we have
\begin{align*}
\Prob\biggl\{ \sup_{t  \in \R} \sqrt{n}\left\lbrace   \empiricalCDF(t)-\populationCDF(t)\right\rbrace > \xi \biggr\} \leq \exp\biggl( - 2 \xi^{2} - \frac{\zeta^2}{n}\biggl\{1+\frac{4\zeta}{5n}+\frac{425 \zeta^2}{525n^2 }\biggr\} \biggr).
\end{align*}
\end{corollary}

In comparison, \citet[Theorem 1]{massart1990tight} gives a slightly larger bound of $\exp(-2\xi^2)$ for  $\xi \geq \min\{\sqrt{\log(2)/2}, 1.0841 n^{-1/6}\}$. \citet{massart1990tight} remarks that whilst this captures the regime of interest for statistical purposes, the proof therein does not extend to smaller values of $\xi$. Corollary \ref{corr:interpretable_dkw_m_inequality} verifies the conjecture of \cite{birnbaum1958distribution} for all values of $\xi \geq 0$ and Corollary \ref{corr:dkw_m_inequality_refined_pinsker_strengthening} proves a slightly stronger result.

Our final corollary is a variance-adaptive bound for the entire real line, at the expense of an extra logarithmic factor. 

\begin{corollary}\label{corr:bartl_mendelson_type} Suppose that $\delta \in (0,1)$, $\beta>1$ and let $\varepsilon_\beta(n,\delta):=\log(2\lceil\log_\beta(n)\rceil/\delta)/n$. Then,
\begin{align*}
\Prob\biggl\{\empiricalCDF(t)-\populationCDF(t) >  \biggl(\beta \interpretableMassartDudleyTypeBoundPointwise\bigl(\populationCDF(t),\varepsilon_\beta(n,\delta)\bigr)\sqrt{{\varepsilon_\beta(n,\delta)/2}}\biggr)\vee \biggl(\frac{1}{n}\biggr) \text{ for some }t \in \R \biggr\} \leq \delta.
\end{align*}
\end{corollary}

Corollary \ref{corr:bartl_mendelson_type} is a slight improvement on recent results from \citet[Theorem 1.2]{bartl2023variance} and \citet[Theorem 5]{blanchard2023tight} where similar bounds are given but the universal constants are left unspecified. We remark that Corollary \ref{corr:bartl_mendelson_type} can be restated as a simultaneous lower-confidence bound by taking 
\begin{align*}
\mathcal{U}(q,\varepsilon):=\frac{3 \beta \bigl( (2q -1)(3\beta-1)\varepsilon + \sqrt{\epsilon \{18 \sigma^2(q) + \varepsilon(3 \beta-1)^2\}}\bigr)}{9+2\varepsilon(3\beta-1)^2 }.
\end{align*}
Indeed, Corollary \ref{corr:bartl_mendelson_type} may be restated as follows
\begin{align*}
\Prob\biggl\{\populationCDF(t) >\empiricalCDF(t)- \mathcal{U}\bigl(\empiricalCDF(t),\varepsilon_\beta(n,\delta)\bigr)\vee \biggl(\frac{1}{n}\biggr) \text{ for all }t \in \R \biggr\} \geq 1- \delta.
\end{align*}

The proof of Theorem \ref{sec:proofOfMainTheorem} is given in Section \ref{sec:proofOfMainTheorem}, with the proof 
a technical lemma and the corrollaries given in Section \ref{sec:ProofProofOfTechnicalLemma}.


\section{Proof of Theorem \ref{thm:main_result_localised_dkw_m_inequality}}\label{sec:proofOfMainTheorem}

To begin the proof we first introduce a reverse filtration $(\filtrationSigmaAlgebra_t)_{t \in \R}$ (a collection of sigma algebras such that $\filtrationSigmaAlgebra_{t_1} \subseteq \filtrationSigmaAlgebra_{t_0}$ whenever $t_0 \leq t_1$). For each $t \in \R$ we let  $\filtrationSigmaAlgebra_t$  denote the sigma-algebra by all events of the form $\{X_i \leq s\}$ with $s \geq t$ and $i \in [n]$ are measurable.

\newcommand{\tmin}{t_0}

Given $\lambda \geq 0$ let $(M_{\lambda}(t))_{t \geq \tmin}$ denote the non-negative stochastic process  defined by 
\begin{align*}
M_{\lambda}(t):= \frac{1}{(1+\lambda)^n}\left( 1+ \frac{\lambda}{\populationCDF(t)}\right)^{n \empiricalCDF(t)}.
\end{align*}

\begin{lemma}\label{lemma:localDvoretzkyKieferWolfowitzMassartInequalityMartingale} Suppose $\tmin \in \R$ satisfies $\populationCDF(\tmin)>0$. Then $(M_{\lambda}(t))_{t \geq \tmin}$ is a reverse martingale with respect to  $(\filtrationSigmaAlgebra_t)_{t \geq \tmin}$. Moreover, $\E\{M_{\lambda}(t)\}=1$ for all $t \geq \tmin$.
\end{lemma}
\begin{proof} Take $s_0 \in [\tmin,\infty)$ and $s_1 \in (s_0,\infty)$. Observe that, conditionally on $\filtrationSigmaAlgebra_{s_1}$, the random variable $W :=n\empiricalCDF(s_0)$ has the distribution of a Binomial random variable with $m:=n\empiricalCDF(s_1)$ trials and success probability $p:=\populationCDF(s_0)/\populationCDF(s_1)$. Thus, for each $\lambda \geq 0$ we have
\begin{align*}
\E\left\lbrace \left( 1+ \frac{\lambda}{\populationCDF(s_0)}\right)^{n \empiricalCDF(s_0)} ~\bigg|~ \filtrationSigmaAlgebra_{s_1}\right\rbrace & =  \E\left\lbrace \left( 1+ \frac{\lambda}{ p \populationCDF(s_1)}\right)^{W} ~\bigg|~ \filtrationSigmaAlgebra_{s_1}\right\rbrace \\
& = \left\lbrace p \left( 1+ \frac{\lambda}{ p \populationCDF(s_1)}\right)+(1-p)\right\rbrace ^{m}= \left( 1+ \frac{\lambda}{\populationCDF(s_1)}\right)^{n \empiricalCDF(s_1)}.
\end{align*}
which entails the reverse martingale property. Applying the dominated convergence theorem with the bound $ \bigl( 1+ {\lambda}/\populationCDF(s_0)\bigr)^{n}$ yields that for all $t \geq \tmin$, 
\begin{align*}\E\{M_{\lambda}(t)\} = \lim_{s \nearrow \infty}\E\{M_{\lambda}(s)\}=\E\bigg\{\lim_{s \nearrow \infty}M_{\lambda}(s)\bigg\}=1.\end{align*}\end{proof}

\newcommand{\firstArgErrorFunctionZero}{r}

Next, for each $\varepsilon >0$, we define a function $\errorFunctionZero \equiv \errorFunctionZero_\varepsilon: (0,1]\times (0,\infty) \rightarrow [0,\infty)$ by  
\begin{align*}
\errorFunctionZero(\firstArgErrorFunctionZero,\lambda)\equiv \errorFunctionZero_{\varepsilon}(\firstArgErrorFunctionZero,\lambda)&:=  \frac{\log(1+\lambda)+\varepsilon}{ \log\left(1+\lambda/\firstArgErrorFunctionZero\right)}-\firstArgErrorFunctionZero.
\end{align*}
Note that the inequality $(1+\lambda/\firstArgErrorFunctionZero)^\firstArgErrorFunctionZero \leq (1+\lambda)$ for $\firstArgErrorFunctionZero \in [0,1]$ ensures that $\errorFunctionZero$ is non-negative. Given $\lambda>0$ and $\delta \in (0,1)$ we define \[\event(\lambda,\delta):=\bigcap_{t \in \R : \populationCDF(t) >0 }\bigl\{\empiricalCDF(t) -\populationCDF(t)\leq  \errorFunctionZero_{\epsByNDelta}(\populationCDF(t),\lambda)\bigr\}.\]

\begin{lemma}\label{lemma:localDvoretzkyKieferWolfowitzMassartInequalityGeneral} For all $\lambda>0$ and $\delta \in (0,1)$ we have $\Prob\{\event(\lambda,\delta)\} \geq 1-\delta$.
\end{lemma}

\newcommand{\T}{\mathbb{T}}
\newcommand{\Z}{\mathbb{Z}}

\sloppy
\begin{proof} Given $k \in \N$ we let $\T_k:=\{ t \in [-2^k,2^k]: t=2^{-k}\lfloor 2^k t \rfloor \text{ and } \populationCDF(t)>0 \}$. By Lemma \ref{lemma:localDvoretzkyKieferWolfowitzMassartInequalityMartingale} the process $\{M_{\lambda}(t)\}_{t \in \T_k}$ is a reverse martingale with $\E(M_{\lambda}(t))=1$ for all $k \in \N$ and $t \in \T_k$. Hence, by Doob's maximal inequality \cite[Theorem 10.4.2]{dudley2018real} we have $\Prob\{\event_k(\lambda,\delta)\} \geq 1-\delta$ where
\begin{align*}
\event_k(\lambda,\delta):=\biggl\{ \max_{t \in \T_k} M_{\lambda}(t) < {1}/{\delta}\biggr\} = \bigcap_{t \in \T_k}\left\lbrace \empiricalCDF(t) < \populationCDF(t)+ \errorFunctionZero_{\epsByNDelta}\bigl(\populationCDF(t),\lambda \bigr) \right\rbrace,
\end{align*}
where the equality between the events follows by taking logarithms. In addition, since $t \mapsto  M_{\lambda}(t)$ is almost surely right-continuous and $\bigcup_{k \in \N}\T_k$ is dense in $\{t \in \R : \populationCDF(t)>0\}$ we have $\event(\lambda,\delta)= \bigcap_{k \in \N} \event_k(\lambda,\delta)$. Thus, since the events $\{\event_k(\lambda,\delta)\}_{k \in \N}$ form a decreasing sequence we have  $\Prob\{\event(\lambda,\delta)\} \geq 1-\delta$.
\end{proof}

\begin{lemma}\label{lemma:technicalLemma} Given any non-empty closed interval $\probabilityInterval \subseteq  (0,e^{-\varepsilon})$ there exists  $(\fixedPointP,\fixedPointLambda) \in \probabilityInterval \times (0,\infty)$ with $\errorFunctionZero(\fixedPointP,\fixedPointLambda) = \inf_{\lambda \in (0,\infty)} \errorFunctionZero(\fixedPointP,\lambda) = \sup_{p \in \probabilityInterval} \errorFunctionZero(p,\fixedPointLambda)=\modulusOfContinuityKL(\fixedPointP,\varepsilon)$. 
\end{lemma}

The proof of Lemma \ref{lemma:technicalLemma} is somewhat technical and is given in Section \ref{sec:ProofProofOfTechnicalLemma}.

\begin{proof}[Proof of Theorem \ref{thm:main_result_localised_dkw_m_inequality}] Choose $\zeta \in (0,1)$ and suppose that the closed interval $\probabilityInterval_\zeta:= \topologicalClosure(\{\populationCDF(t): t \in \intervalReals \}) \cap [\zeta,e^{-\varepsilon}-\zeta]$ is non-empty. By applying Lemma \ref{lemma:technicalLemma} with $\varepsilon= \epsByNDelta$ we obtain $(\fixedPointP,\fixedPointLambda) \in \probabilityInterval \times (0,\infty)$ so that $\errorFunctionZero(\fixedPointP,\fixedPointLambda) = \sup_{p \in \probabilityInterval_\zeta} \errorFunctionZero(p,\fixedPointLambda) = \modulusOfContinuityKL(\fixedPointP,\varepsilon)> 0$. Moreover, by Lemma \ref{lemma:localDvoretzkyKieferWolfowitzMassartInequalityGeneral} we have $\Prob\{\event(\fixedPointLambda,\delta)\} \geq 1-\delta$. On the event $\event(\fixedPointLambda,\delta)$, for each $t \in \intervalReals$ with  $\populationCDF(t) \in  [\zeta,e^{-\varepsilon}-\zeta]$ we have
\begin{align*}
\empiricalCDF(t)-\populationCDF(t) &< \errorFunctionZero(\populationCDF(t),\lambda) \leq \sup_{p \in \probabilityInterval_\zeta} \errorFunctionZero(p,\fixedPointLambda)
 = \errorFunctionZero(\fixedPointP,\fixedPointLambda) =  \modulusOfContinuityKL(\fixedPointP,\varepsilon) \\
&\leq \sup_{p \in \topologicalClosure(\{\populationCDF(t): t \in \intervalReals \})} \modulusOfContinuityKL(p,\varepsilon) = \sup_{p \in \{\populationCDF(t): t \in \intervalReals \}} \modulusOfContinuityKL(p,\varepsilon) = \sup_{t \in \intervalReals} \modulusOfContinuityKL\bigl(\populationCDF(t),\varepsilon\bigr),
\end{align*}
where we applied the continuity of $p \mapsto \modulusOfContinuityKL(p,\varepsilon)$ (Lemma \ref{lemma:checkContinuityModulusIsCont}) for the penulitmate equality. Note that the bound holds vacuously if $\probabilityInterval_\zeta=\emptyset$. Moreover, if $t \in \intervalReals$ satisfies $\populationCDF(t) \in \{0\}\cup [e^{-\varepsilon},1]$ then we have $\empiricalCDF(t)- \populationCDF(t) \leq \modulusOfContinuityKL(\populationCDF(t), \varepsilon) \leq  \sup_{t \in \intervalReals} \modulusOfContinuityKL(\populationCDF(t),\varepsilon)$ almost surely (Lemma \ref{lemma:probabilityOneEitherExtreme}). Hence, 
\begin{align*}
\Prob\biggl\{   \empiricalCDF(t)-\populationCDF(t) > \sup_{t \in \intervalReals}\,\modulusOfContinuityKL\bigl(\populationCDF(t),\epsByNDelta\bigr)\text{ for some } t  \in \intervalReals \text{ with } \populationCDF(t) \notin (0,\zeta) \cup (e^{-\varepsilon}-\zeta)\biggr\} \leq \delta.
\end{align*}
Since this bound holds for all sufficiently small $\zeta$, the conclusion of the theorem follows by continuity of probability. 
\end{proof}

\section{Proofs of technical lemmas}\label{sec:ProofProofOfTechnicalLemma}

\newcommand{\twoVariableFunction}{\Gamma}
\newcommand{\implicitOneVariableFunction}{\gamma}
\newcommand{\Y}{\mathcal{Y}}

Throughout we shall consider $[-\infty, \infty] :=\R \cup \{-\infty,\infty\}$ with the usual topology, generated by the topological basis consisting of all intervals of the form $(a,b)$, $(a,\infty]$ or $[-\infty,b)$ where $a,b\in \R$. We shall use the following elementary form of the implicit function theorem (Lemma \ref{lemma:generalisedImplicitFunctionTheorem}). Given a Borel set $\Y \subseteq [-\infty,\infty]$ we shall say that a function $g: \Y \rightarrow [-\infty,\infty]$ is \emph{sign-monotonic} if 
\begin{align}\label{eq:signMonotonicityCondition}
(g(y_a),g(y_b)) \in  ([-\infty,0) \times [-\infty,\infty]) \cup ([-\infty,\infty] \times (0,\infty]),
\end{align}
for all $y_a, y_b \in \Y$ with $y_a<y_b$. Note that a sufficient condition for  $g: \Y \rightarrow [-\infty,\infty]$ to be sign-monotonic is that $g$ is both non-decreasing and $|\{y \in \Y:g(y)=0\}| \leq 1$.

\begin{lemma}\label{lemma:generalisedImplicitFunctionTheorem} Let $\X \subseteq \R^d$ be a Borel set, $\Y \subseteq [-\infty,\infty]$ is a non-empty interval which is bounded from below and let $\twoVariableFunction:\X \times \Y \rightarrow [-\infty,\infty]$ be a function. Suppose that for each $x \in \X$ the function $y \mapsto \twoVariableFunction(x,y)$ is sign-monotonic. Suppose further that for each $y \in \Y$, the function $x\mapsto \twoVariableFunction(x,y)$ is continuous and define a function $\implicitOneVariableFunction : \X \rightarrow \topologicalClosure(\Y)$ by 
\begin{align*}
\implicitOneVariableFunction(x):= \begin{cases} \sup\Y & \text{ if }  \twoVariableFunction(x,y) < 0  \text{ for all } y \in \Y\\
\inf\{y \in \Y : \twoVariableFunction(x,y)\geq 0\} &\text{ otherwise.}
\end{cases}
\end{align*}
Then $\implicitOneVariableFunction$ is continuous and $\twoVariableFunction(x,y) (y- \implicitOneVariableFunction(x))>0$ for all $x \in \X$ and $y \in \Y \setminus\{\implicitOneVariableFunction(x)\}$.
\end{lemma}

Whilst we expect that Lemma \ref{lemma:generalisedImplicitFunctionTheorem} is known, we include a proof in Section \ref{sec:proofOfTopologicalImplicitFunctionThm} for the convenience of the reader.

\begin{lemma}\label{lemma:modulusContinuityWellDefinedByKLProperty} For each $p \in (0,1)$,  the function $\eta \mapsto \kullbackLeibler(p+\eta ||p)$ is strictly increasing and continuous bijection from $[0,1-p]$ to $[0,\log(1/p)]$. Hence, for $p \in (0,e^{-\varepsilon}]$, there is a unique value of $\eta \in [0,1-p]$ with $\kullbackLeibler(p+\eta ||p)=\varepsilon$ and if $p \in (e^{-\varepsilon},1]$ then $\kullbackLeibler(p+\eta||p)< \varepsilon$ for all $\eta \in [0,1-p]$.
\end{lemma}

\begin{proof} Given $p \in (0,1)$ and $\eta_0 \in [0,1-p)$, we have $p+\eta_0<1$ so the map $\eta \mapsto \kullbackLeibler(p+\eta ||p)$ is differentiable at $\eta_0$, with derivative
\begin{align*}
\log{\left(\frac{\eta_0 + p}{p} \right)} + \log{\left(\frac{1 - p}{1-p- \eta_0} \right)}>0.
\end{align*} 
Hence, $\eta \mapsto \kullbackLeibler(p+\eta ||p)$ is also strictly increasing and continuous function on $[0,1-p)$. Moreover, since $p \in (0,1)$ we have $\lim_{a \nearrow 1} \kullbackLeibler(a||p)=\log(1/p)=\kullbackLeibler(1||p)$ so that $\eta \mapsto \kullbackLeibler(p+\eta ||p)$ is strictly increasing and continuous bijection from $[0,1-p]$ to $[0,\log(1/p)]$. Finally, note that $\varepsilon \in [0,\log(1/p)] = \{ \kullbackLeibler(p+\eta||p):\eta \in [0,1-p]\}$ if and only if $p \in (0,e^{-\varepsilon})$.
\end{proof}

\begin{lemma}\label{lemma:checkContinuityModulusIsCont} Given $\varepsilon >0$ the function $p\mapsto \modulusOfContinuityKL(p,\varepsilon)$ is continuous.
\end{lemma}

\begin{proof} Let's define a function $\twoVariableFunction_{\kullbackLeibler}: [0,1]^2 \rightarrow [0,\infty]$ by
\begin{align*}
\twoVariableFunction_{\kullbackLeibler}(x,y) := \begin{cases} -\varepsilon & \text{ if } y\leq x\\
\kullbackLeibler(y||x) - \varepsilon &\text{ otherwise. }
\end{cases}
\end{align*}
For each $x \in (0,1)$, it follows from Lemma \ref{lemma:modulusContinuityWellDefinedByKLProperty} that $y \mapsto \twoVariableFunction_{\kullbackLeibler}(x,y)$ is sign-monotonic. Furthermore, if $x \in \{0,1\}$ then $y\mapsto \twoVariableFunction_{\kullbackLeibler}(x,y)$ is a non-decreasing and piecewise constant with $\{ y \in [0,1]: \twoVariableFunction_{\kullbackLeibler}(x,y) = 0\}=\emptyset$. As such,  $y \mapsto \twoVariableFunction_{\kullbackLeibler}(x,y)$ is sign-monotonic for all $x \in [0,1]$. Note also that for each $y \in [0,1]$, the map $x \mapsto \kullbackLeibler(y||x)$ is continuous; and consequently so is $x \mapsto \twoVariableFunction_{\kullbackLeibler}(x,y)$. Thus, by Lemma \ref{lemma:generalisedImplicitFunctionTheorem} the function $\implicitOneVariableFunction_{\kullbackLeibler}: [0,1] \rightarrow [0,1]$ by 
\begin{align*}
\implicitOneVariableFunction_{\kullbackLeibler}(x):= \begin{cases} 1 & \text{ if }  \kullbackLeibler(y||x) < \varepsilon  \text{ for all } y \in [0,1]\\
\inf\{y \in [0,1] : \kullbackLeibler(y||x) \geq  \varepsilon\} &\text{ otherwise.}
\end{cases}
\end{align*} 
is also continuous. Finally note that for all $p \in [0,1]$ we have $\modulusOfContinuityKL(p,\varepsilon) = \implicitOneVariableFunction_{\kullbackLeibler}(p)-p$, so that $p\mapsto \modulusOfContinuityKL(p,\varepsilon)$ is also continuous.     
\end{proof}

\begin{proof}[Proof of Lemma \ref{lemma:technicalLemma}] Observe that $\errorFunctionZeroRestriction$ is twice differentiable on $(0,1]\times (0,\infty)$ with
\begin{align*}
{\partial_\firstArgErrorFunctionZero\errorFunctionZeroRestriction(\firstArgErrorFunctionZero,\lambda)} & =\frac{\lambda \left\lbrace\varepsilon + \log{\left(1+\lambda \right)}\right\rbrace}{\firstArgErrorFunctionZero \left(\lambda + \firstArgErrorFunctionZero\right) \log^2{\left(1+{\lambda}/{\firstArgErrorFunctionZero}  \right)}}-1,\\
{\partial_{\firstArgErrorFunctionZero,\firstArgErrorFunctionZero}\errorFunctionZeroRestriction(\firstArgErrorFunctionZero,\lambda)}&=
\frac{\lambda^2 \left\lbrace\varepsilon + \log{\left(1+\lambda \right)}\right\rbrace \left\lbrace 2 - \left(1+2\firstArgErrorFunctionZero/\lambda\right) \log{\left(1+{\lambda}/{\firstArgErrorFunctionZero}  \right)} \right\rbrace}{\firstArgErrorFunctionZero^2 \left(\firstArgErrorFunctionZero+{\lambda} \right)^2 \log^3{\left(1+{\lambda}/{\firstArgErrorFunctionZero}  \right)}}.
\end{align*}
By applying the inequality $\log(1+z)> 2z/(2+z)$ for $z>0$ we deduce $\partial_{\firstArgErrorFunctionZero,\firstArgErrorFunctionZero} \errorFunctionZeroRestriction(\firstArgErrorFunctionZero,\lambda)<0$ for $\firstArgErrorFunctionZero>0$, so $\firstArgErrorFunctionZero \mapsto \errorFunctionZeroRestriction(\firstArgErrorFunctionZero,\lambda)$ is strictly concave. Now consider the function $\twoVariableFunctionDerivPFunction: (0,\infty) \times \probabilityInterval \rightarrow [-\infty,\infty]$ defined by $\twoVariableFunctionDerivPFunction(\lambda,\firstArgErrorFunctionZero):= -
{\partial_\firstArgErrorFunctionZero\errorFunctionZeroRestriction(\firstArgErrorFunctionZero,\lambda)}$. For each $\lambda \in (0,\infty)$, the strict concavity of $\firstArgErrorFunctionZero \mapsto \errorFunctionZeroRestriction(\firstArgErrorFunctionZero,\lambda)$ implies that $\firstArgErrorFunctionZero \mapsto \twoVariableFunctionDerivPFunction(\lambda,\firstArgErrorFunctionZero)$ is strictly increasing and as such $|\{ \firstArgErrorFunctionZero \in \probabilityInterval : \twoVariableFunctionDerivPFunction(\lambda,\firstArgErrorFunctionZero)=0\}| \leq 1$. Moreover, for each $\firstArgErrorFunctionZero \in \probabilityInterval$ the function $\lambda \mapsto 
\twoVariableFunctionDerivPFunction(\lambda,\firstArgErrorFunctionZero)$ is continuous. Hence, by Lemma \ref{lemma:generalisedImplicitFunctionTheorem} there exists a continuous function  $\pFunctionOfLambda: \left(0,\infty\right) \rightarrow \probabilityInterval$ such that $(\pFunctionOfLambda(\lambda) - \firstArgErrorFunctionZero)\partial_\firstArgErrorFunctionZero\errorFunctionZeroRestriction(\firstArgErrorFunctionZero,\lambda)<0$ for all $\lambda \in (0,\infty)$ and $p \in \probabilityInterval \setminus \{\pFunctionOfLambda(\lambda)\}$. Hence, we have 
$\errorFunctionZero(\firstArgErrorFunctionZero,\lambda) \leq \errorFunctionZero(\pFunctionOfLambda(\lambda),\lambda)$ for all $(\firstArgErrorFunctionZero,\lambda) \in \probabilityInterval \times \left(0,\infty\right)$.  

Next, we define a function $\twoVariableFunctionDerivLambdaFunction: \probabilityInterval \times (0,\infty) \rightarrow \R$ by  
\begin{align*}
\twoVariableFunctionDerivLambdaFunction(p,\lambda)&:= \left(p+\lambda \right) \log{\left(1+\lambda/p \right)}-\left(1+\lambda \right) \left\lbrace \varepsilon + \log{\left(\lambda + 1 \right)}\right\rbrace,
\end{align*}
so that $\partial_{\lambda}\errorFunctionZero(p,\lambda) =\twoVariableFunctionDerivLambdaFunction(p,\lambda)/\left\lbrace \left(1+\lambda\right) \left(p+\lambda \right) \log^2{\left(1+{\lambda}/{p}  \right)}\right\rbrace$. Note that for each $\lambda \in (0,\infty)$ the map $p \mapsto \twoVariableFunctionDerivLambdaFunction(p,\lambda)$ is continuous on $\probabilityInterval$ with
\begin{align*}
\partial_{\lambda}\twoVariableFunctionDerivLambdaFunction(p,\lambda)&=   \log{\left(1+\lambda/p  \right)}-\{\varepsilon + \log{\left(1 +\lambda \right)}\},
\end{align*}
so $\twoVariableFunctionDerivLambdaFunction(p,\lambda)=0$ implies $\partial_{\lambda}\twoVariableFunctionDerivLambdaFunction(p,\lambda)>0$. Thus, $\lambda \mapsto \twoVariableFunctionDerivLambdaFunction(p,\lambda)$ is sign-monotonic for all $p \in \probabilityInterval$. Consequently, there exists a continuous function $\lambdaFunctionOfP: \probabilityInterval \rightarrow [0,\infty]$ such that $\twoVariableFunctionDerivLambdaFunction(p,\lambda)(\lambdaFunctionOfP(p)-\lambda)>0$, and consequently, $\partial_{\lambda}\errorFunctionZero(p,\lambda) (\lambdaFunctionOfP(p)-\lambda)>0$ for all $p \in \probabilityInterval$ and $\lambda \in (0,\infty) \setminus \{\lambdaFunctionOfP(p)\}$.  Note also that for each $p \in \probabilityInterval \subseteq (0,e^{-\varepsilon}) $ we have $\lim_{\lambda \searrow 0}\twoVariableFunctionDerivPFunction(p,\lambda)=-\varepsilon<0$ and $\lim_{\lambda \nearrow \infty}\twoVariableFunction(p,\lambda)=\infty$ so by by the continuity of $\lambda \mapsto \twoVariableFunction(p,\lambda)$ there must be at least one root in $(0,\infty)$, which must coincide with $\lambdaFunctionOfP(p)$ i.e.  $\twoVariableFunction(p,\lambdaFunctionOfP(p))=0$. Hence, $\lambdaFunctionOfP(p) \in (0,\infty)$ and $\errorFunctionZero(p,\lambdaFunctionOfP(p))\leq \errorFunctionZero(p,\lambda)$ for all $\lambda \in (0,\infty)$. 

Next, we define a function $\Phi:\probabilityInterval \rightarrow \R$ by $\Phi(\rho) = \pFunctionOfLambda\left( \lambdaFunctionOfP(\rho)\right)-\rho$ and choose $\pMin,\pMax \in (0,e^{-\varepsilon})$ so that $\probabilityInterval=[\pMin,\pMax]$. Notice that $\Phi$ is continuous with $ \Phi(\pMin)\geq 0$ and $ \Phi(\pMax)\leq 0$ so by the intermediate value theorem there must exist some $\fixedPointP \in \probabilityInterval$ with $\Phi(\fixedPointP)=0$. Letting $\fixedPointLambda:= \lambdaFunctionOfP(\fixedPointP)$ we obtain a pair $(\fixedPointP,\fixedPointLambda) \in \probabilityInterval \times (0,\infty)$ satisfying $\errorFunctionZero(\fixedPointP,\fixedPointLambda) = \inf_{\lambda \in (0,\infty)} \errorFunctionZero(\fixedPointP,\lambda) = \sup_{p \in \probabilityInterval} \errorFunctionZero(p,\fixedPointLambda)$.

In addition, given any $p \in (0,e^{-\varepsilon})$ we have $\twoVariableFunctionDerivLambdaFunction(p,\lambdaFunctionOfP(p))=0$ which yields $p+ \errorFunctionZero\{p,\lambdaFunctionOfP(p)\}=(p+\lambdaFunctionOfP(p))/(1+\lambdaFunctionOfP(p)) \in (p,1)$ and $\kullbackLeibler\left( p+ \errorFunctionZero\{p,\lambdaFunctionOfP(p)\}||p\right) = \varepsilon$ upon rearranging. Hence, $\modulusOfContinuityKL(p,\lambdaFunctionOfP(p))=\errorFunctionZero(p,\varepsilon)>0$. Since $\fixedPointLambda= \lambdaFunctionOfP(\fixedPointP)$ this completes the proof of the lemma.
\end{proof}

For the next lemma we define an event $\event_0$ by
\begin{align*}
\event_0:=\bigcap_{t \in \R : \populationCDF(t) \in \{0\}\cup [e^{-\varepsilon},1]}\left\lbrace \empiricalCDF(t)- \populationCDF(t) \leq \modulusOfContinuityKL(\populationCDF(t), \varepsilon)\right\rbrace. 
\end{align*}

\begin{lemma}\label{lemma:probabilityOneEitherExtreme} We have $\Prob(\event_0)=1$.
\end{lemma}
\begin{proof} Let's introduce $J_0:=\{ t \in \R : \populationCDF(t) =0 \}$ and $\tilde{\event}_0:= \{ \sup_{t \in \R} \empiricalCDF(t) \leq 1\} \cap  \bigcap_{i=1}^n\{X_i \notin J_0\}$, so that $\Prob(\tilde{\event}_0)=1$. Now suppose the event $\tilde{\event}_0$ holds. If $\populationCDF(t)=0$ then $(0,t] \subseteq J_0$ so that $\empiricalCDF(t)=0$, and so $\empiricalCDF(t)-\populationCDF(t) = 0 = \modulusOfContinuityKL(\populationCDF(t),\varepsilon)$. On the other hand, if $\populationCDF(t) \geq e^{-\varepsilon}$ then  $\empiricalCDF(t) - \populationCDF(t) \leq 1- \populationCDF(t)= \modulusOfContinuityKL(\populationCDF(t),\varepsilon)$, where we have used the definition of $\modulusOfContinuityKL(p,\varepsilon)$ for $\log(1/p) <\varepsilon$ and the continuity of $p\mapsto \modulusOfContinuityKL(p,\varepsilon)$ at $p = e^{-\varepsilon}$ (Lemma \ref{lemma:checkContinuityModulusIsCont}).   
\end{proof}

Given  $(p,\eta) \in (0,1)$ and $\eta \in [0,1-p]$ we let
\begin{align*}
\massartKLLowerBoundFunction(p,\eta):= \frac{\eta^2}{2(p+\eta/3)(1-p-\eta/3)}.
\end{align*}

\begin{lemma}\label{lemma:dudleyMassartKLLBLemma} For all $p \in (0,1)$ and $\eta \in [0,1-p]$ we have $\massartKLLowerBoundFunction(p,\eta)\leq \kullbackLeibler(p+\eta||p)$.
\end{lemma}
\begin{proof} Given $z \in (0,1]$ let's write $f_z:(0,z] \rightarrow \R$ for the function
\begin{align*}
f_z(\eta)&:=\kullbackLeibler(z||z-\eta)-\massartKLLowerBoundFunction(z-\eta,\eta).
\end{align*}
Observe that $f_z(0)=0$ and for $\eta <z \leq 1$ we have
\begin{align*}
\partial_\eta f_z(\eta) =  \frac{\eta^{3} \left(16 \eta^{2} - 42 \eta z + 21 \eta + 27 z^{2} - 27 z + 9\right)}{(3z- 2 \eta )^{2}  (3+ 2 \eta - 3z)^{2}(z- \eta) (1+ \eta - z)}>0.
\end{align*}
Hence, $\kullbackLeibler(p+\eta||p)=\massartKLLowerBoundFunction(p,\eta)+f_{p+\eta}(\eta) \geq \massartKLLowerBoundFunction(p,\eta)$, which establishes the first part of the lemma.
\end{proof}

\section{Proofs of the corollaries}\label{sec:proofOfCorollaries}

\begin{proof}[Proof of Corollary \ref{corr:interpretable_dkw_m_inequality}] Let $\pMin:=\inf_{t \in \intervalReals} \populationCDF(t)$, $\pMax:=\sup_{t \in \intervalReals} \populationCDF(t)$ and take $\varepsilon=\epsByNDelta$. Given $t \in \intervalReals$ we have $p=\populationCDF(t) \in [\pMin,\pMax]$ we have  $\kullbackLeibler(p+\eta||p)< \varepsilon$ for any $\eta \in (0,\modulusOfContinuityKL(p,\varepsilon))$. Thus, by Lemma \ref{lemma:dudleyMassartKLLBLemma} we have $\massartKLLowerBoundFunction(p,\eta) < \varepsilon$. By rearranging this bound with $\probByEps=9/(9+2\varepsilon)$, so $\sigma^2(\probByEps) / \probByEps^2 = 2\varepsilon/9$ we deduce that 
\begin{align*}
\eta&< \frac{3 ( \sqrt{\varepsilon \{\varepsilon +18p(1-p)\}}+ \varepsilon \{1-2 p\})}{2 \varepsilon + 9} =\probByEps \biggl\{ \sqrt{4p(1-p)+\frac{2\varepsilon}{9}}+ \bigl(1-2 p\bigr)\sqrt{\frac{2\varepsilon}{9}}\biggr\}\sqrt{\frac{\varepsilon}{2}}\\
& = \biggl\{ \sqrt{4\sigma^2(p)\probByEps^2+\sigma^2(\probByEps)}+ \bigl(1-2 p\bigr)\sigma(\probByEps)\biggr\}\sqrt{\frac{\varepsilon}{2}}=\interpretableMassartDudleyTypeBoundPointwise(p,\varepsilon) \sqrt{\frac{\varepsilon}{2}}.
\end{align*}
Since this result for all $t \in \intervalReals$ and $\eta \in (0,\modulusOfContinuityKL(p,\varepsilon))$ we deduce that 
\begin{align*}
\sup_{t \in \intervalReals}\,\modulusOfContinuityKL\bigl(\populationCDF(t),\varepsilon\bigr) \leq  \sup_{p \in [\pMin,\pMax]}\interpretableMassartDudleyTypeBoundPointwise(p,\varepsilon)\sqrt{\frac{\varepsilon}{2}}.
\end{align*}
Moreover, the function $p \mapsto \interpretableMassartDudleyTypeBoundPointwise(p,\varepsilon)$ is unimodal, strictly increasing on $[0, 1/2-\sqrt{\varepsilon/3}]$ and strictly decreasing on $[1/2-\sqrt{\varepsilon/3},1]$. Thus, we have  \[\interpretableMassartDudleyTypeBoundInterval(\intervalReals,\varepsilon)= \sup_{p \in [\pMin,\pMax]}\interpretableMassartDudleyTypeBoundPointwise(p,\varepsilon) \geq  \sup_{t \in \intervalReals}\,\modulusOfContinuityKL\bigl(\populationCDF(t),\varepsilon\bigr)\sqrt{\frac{2}{\varepsilon}},\] and the conclusion of Corollary \ref{corr:interpretable_dkw_m_inequality} follows from Theorem \ref{thm:main_result_localised_dkw_m_inequality}.    
\end{proof}

\begin{proof}[Proof of Corollary \ref{corr:dkw_m_inequality_refined_pinsker_strengthening}] Given $\xi\geq 0$, $\zeta:=2\xi^2/3$, $\eta = \xi/\sqrt{n}$ and $p \in [0,1-\eta]$ it follows from \citet[Theorem 7]{fedotov2003refinements} that
\begin{align*}
\kullbackLeibler\bigl(p+\eta||p\bigr) & \geq  2 \eta^2 +\frac{4\eta^4}{9}+\frac{32\eta^6}{135}+\frac{7072\eta^8}{42525} >   n \biggl(2 \xi^{2} + \frac{\zeta^2}{n}\biggl\{1+\frac{4\zeta}{5n}+\frac{425 \zeta^2}{525n^2 }\biggr\}\biggr)=:\log(1/\delta). 
\end{align*}
Thus, $\modulusOfContinuityKL\bigl(p,\epsByNDelta\bigr) \leq \eta$ for all $p \in [0,1]$. By Theorem \ref{thm:main_result_localised_dkw_m_inequality} we deduce that 
\begin{align*}
\Prob\biggl\{ \sup_{t  \in \intervalReals}\bigl\{  \empiricalCDF(t)-\populationCDF(t)\bigr\} > \eta \biggr\} \leq \delta.
\end{align*} 
Rearranging yields the conclusion of Corollary \ref{corr:dkw_m_inequality_refined_pinsker_strengthening}.
\end{proof}

\begin{proof}[Proof of Corollary \ref{corr:bartl_mendelson_type}] By computing first and second derivatives we note that the map $p\mapsto \interpretableMassartDudleyTypeBoundPointwise(p,\varepsilon) \sqrt{\varepsilon/2}$ is concave with a unique maximum attained at $p_\star:=\{1-\sigma(\probByEps)/\probByEps\}/2$. Hence, we may construct a sequence of closed intervals $\{A_r\}_{r=1}^{\lceil\log_\beta(n)\rceil-1}$ and $\{B_r\}_{r=1}^{\lceil\log_\beta(n)\rceil-1}$ such that $A_r \subseteq [0,p_\star]$, $B_r \subseteq [p_\star,1]$ and 
\begin{align*}
\inf_{p \in A_r}   \interpretableMassartDudleyTypeBoundPointwise\bigl\{\populationCDF(t),\varepsilon_\beta(n,\delta)\bigr\} &= \inf_{p \in B_r}   \interpretableMassartDudleyTypeBoundPointwise\bigl\{\populationCDF(t),\varepsilon_\beta(n,\delta)\bigr\} = \sqrt{\frac{2 \beta^{-2(r+1)}}{\varepsilon_\beta(n,\delta)}},\\
\sup_{p \in A_r}   \interpretableMassartDudleyTypeBoundPointwise\bigl\{\populationCDF(t),\varepsilon_\beta(n,\delta)\bigr\} &= \sup_{p \in B_r}   \interpretableMassartDudleyTypeBoundPointwise\bigl\{\populationCDF(t),\varepsilon_\beta(n,\delta)\bigr\} = \sqrt{\frac{2 \beta^{-2r}}{\varepsilon_\beta(n,\delta)}}.
\end{align*}
By Corollary \ref{corr:interpretable_dkw_m_inequality}, for each $r = 1,\ldots, \lceil\log_\beta(n)\rceil-1$ we have
\begin{align*}
& \Prob\biggl\{   \empiricalCDF(t)-\populationCDF(t) >   \beta \interpretableMassartDudleyTypeBoundPointwise\bigl\{\populationCDF(t),\varepsilon_\beta(n,\delta) \bigr\}\sqrt{\frac{\varepsilon_\beta(n,\delta)}{2}} \text{ for some }t \in \R \text{ with }\populationCDF(t) \in A_r\biggr\}\\
& \leq \Prob\biggl\{ \sup_{t \in \R\, :\, \populationCDF(t)   \in A_r}\bigl\{   \empiricalCDF(t)-\populationCDF(t)\bigr\} >\inf_{p \in A_r}   \beta \interpretableMassartDudleyTypeBoundPointwise\bigl\{p,\varepsilon_\beta(n,\delta)\bigr\} \sqrt{\frac{\varepsilon_\beta(n,\delta)}{2}} \biggr\}\\
& \leq  \Prob\biggl\{ \sup_{t \in \R\, :\, \populationCDF(t)   \in A_r}\bigl\{   \empiricalCDF(t)-\populationCDF(t)\bigr\} >\sup_{p \in A_r}   \interpretableMassartDudleyTypeBoundPointwise\bigl\{p,\varepsilon_\beta(n,\delta)\bigr\} \sqrt{\frac{\varepsilon_\beta(n,\delta)}{2}} \biggr\}\\
& = \Prob\biggl\{ \sup_{t \in \R\, :\, \populationCDF(t)   \in A_r}\bigl\{   \empiricalCDF(t)-\populationCDF(t)\bigr\} >   \interpretableMassartDudleyTypeBoundInterval\bigl\{A_r,\varepsilon_\beta(n,\delta)\bigr\} \sqrt{\frac{\varepsilon_\beta(n,\delta)}{2}} \biggr\} \leq \frac{\delta}{2\lceil\log_\beta(n)\rceil},
\end{align*}
and similarly with $B_r$ in place of $A_r$. We also let $A_{\lceil\log_\beta(n)\rceil}:=[0,\inf A_{\lceil\log_\beta(n)\rceil-1}]$ and $B_{\lceil\log_\beta(n)\rceil}:=[\sup B_{\lceil\log_\beta(n)\rceil-1},1]$ so that 
\begin{align*}
\sup_{p \in A_{\lceil\log_\beta(n)\rceil} \cup B_{\lceil\log_\beta(n)\rceil}}   \interpretableMassartDudleyTypeBoundPointwise\bigl\{\populationCDF(t),\varepsilon_\beta(n,\delta)\bigr\} \leq  \sqrt{\frac{2 \beta^{-2\lceil\log_\beta(n)\rceil}}{\varepsilon_\beta(n,\delta)}} \leq \frac{1}{n} \sqrt{\frac{2 }{\varepsilon_\beta(n,\delta)}} .
\end{align*}
Thus, applying Corollary \ref{corr:interpretable_dkw_m_inequality} once more yields 
\begin{align*}
& \Prob\biggl\{   \empiricalCDF(t)-\populationCDF(t) >   \frac{1}{n} \text{ for some }t \in \R \text{ with }\populationCDF(t) \in A_{\lceil\log_\beta(n)\rceil}\biggr\}\\
& \leq \Prob\biggl\{ \sup_{t \in \R\, :\, \populationCDF(t)   \in A_{\lceil\log_\beta(n)\rceil}}\bigl\{   \empiricalCDF(t)-\populationCDF(t)\bigr\} >  \interpretableMassartDudleyTypeBoundInterval\bigl\{A_{\lceil\log_\beta(n)\rceil},\varepsilon_\beta(n,\delta)\bigr\} \sqrt{\frac{\varepsilon_\beta(n,\delta)}{2}} \biggr\} \leq \frac{\delta}{2\lceil\log_\beta(n)\rceil},
\end{align*}
and similarly with $B_{\lceil\log_\beta(n)\rceil}$ in place of $A_{\lceil\log_\beta(n)\rceil}$. Finally, letting $C_0$ denote the (possibly empty) closed open interval $C_0:= [0,1] \setminus \bigl\{\bigcup_{r=1}^{\lceil\log_\beta(n)\rceil}(A_r\cup B_r)\bigr\}$ we have
\begin{align*}
& \Prob\biggl\{   \empiricalCDF(t)-\populationCDF(t) >   \beta \interpretableMassartDudleyTypeBoundPointwise\bigl\{\populationCDF(t),\varepsilon_\beta(n,\delta) \bigr\}\sqrt{\frac{\varepsilon_\beta(n,\delta)}{2}} \text{ for some }t \in \R \text{ with }\populationCDF(t) \in C_0\biggr\}\\
& = \Prob\biggl\{ \sup_{t \in \R\, :\, \populationCDF(t)   \in C_0}\bigl\{   \empiricalCDF(t)-\populationCDF(t)\bigr\} >\inf_{p \in C_0}   \beta \interpretableMassartDudleyTypeBoundPointwise\bigl\{p,\varepsilon_\beta(n,\delta)\bigr\} \sqrt{\frac{\varepsilon_\beta(n,\delta)}{2}} \biggr\}\\
& = \Prob\biggl\{ \sup_{t \in \R\, :\, \populationCDF(t)   \in C_0}\bigl\{   \empiricalCDF(t)-\populationCDF(t)\bigr\} >1 \biggr\} = 0.
\end{align*}
Hence, the conclusion of Corollary \ref{corr:bartl_mendelson_type} follows by a union bound.
\end{proof}

\section{Proof of the topological implicit function theorem}\label{sec:proofOfTopologicalImplicitFunctionThm}

\begin{proof}[Proof of Lemma \ref{lemma:generalisedImplicitFunctionTheorem}] First let's show that $\twoVariableFunction(x_0,y_0) (y_0- \implicitOneVariableFunction(x_0))>0$ for all $x_0 \in \X$ and $y_0 \in \Y \setminus\{\implicitOneVariableFunction(x)\}$. If $y_0 \in \Y$ and $y_0< \implicitOneVariableFunction(x)$ then $\twoVariableFunction(x_0,y_0)<0$ and so $\twoVariableFunction(x_0,y_0) (y_0- \implicitOneVariableFunction(x_0))>0$. On the other hand, if $y_0 \in \Y$ and $y_0> \implicitOneVariableFunction(x_0)$ then $\implicitOneVariableFunction(x_0) = \inf\{y \in \Y : \twoVariableFunction(x_0,y)\geq 0\} \in \topologicalClosure(\Y)$. Since $\Y$ is an interval we can choose  $y_1 \in  (\implicitOneVariableFunction(x_0),y_0) \subseteq \Y$, so that $\twoVariableFunction(x_0,y_1) \geq 0$ and hence $\twoVariableFunction(x_0,y_0) >0$ since $y \mapsto \twoVariableFunction(x_0,y)$ is sign-monotonic.

We begin by showing that $\implicitOneVariableFunction$ is upper semi-continuous. Fix some $x_0 \in \X$ and $\epsilon>0$. If $\implicitOneVariableFunction(x_0) = \sup \Y$, then $\implicitOneVariableFunction(x_1) \leq \implicitOneVariableFunction(x_0)$ for all $x_1 \in \X$. Suppose then that $\implicitOneVariableFunction(x_0)=\inf\{y \in \Y : \twoVariableFunction(x_0,y)< 0\}<\sup \Y$ and we can choose $y_0 \in (\implicitOneVariableFunction(x_0),\implicitOneVariableFunction(x_0)+\epsilon) \cap \Y$. Since $y_0 > \implicitOneVariableFunction(x_0)$ we have $\twoVariableFunction(x_0,y_0)> 0$. By the continuity of $x\mapsto \twoVariableFunction(x,y_0)$ we deduce that there exists an open set $U \subseteq \R^d$ with $x_0 \in U \cap \X$ and for all $x_1 \in U \cap \X$ we have $\twoVariableFunction(x_1,y_0)>0$. As such, we have $\implicitOneVariableFunction(x_1) \leq y_0<\implicitOneVariableFunction(x_0)+\epsilon$. Thus, $\implicitOneVariableFunction$ is upper semi-continuous.

Next, we show that $\implicitOneVariableFunction$ is lower semi-continuous. Again, we begin by choosing $x_0 \in \X$ and $\epsilon>0$. If $\implicitOneVariableFunction(x_0) = \inf \Y$, then $\implicitOneVariableFunction(x_1) \geq \implicitOneVariableFunction(x_0)$ for all $x_1 \in \X$. Suppose that $\implicitOneVariableFunction(x_0) > \inf \Y$ so that we can choose $y_0 \in (\implicitOneVariableFunction(x_0)-\epsilon,\implicitOneVariableFunction(x_0)) \cap (\inf \Y,\implicitOneVariableFunction(x_0)) \subseteq (\inf \Y , \sup\Y) \subseteq \Y$,
where we have used the fact that $\Y$ is an interval. Since $y_0< \implicitOneVariableFunction(x_0)$ we have $\twoVariableFunction(x_0,y_0)<0$. Hence, by the continuity of $x\mapsto \twoVariableFunction(x,y_0)$, there exists an open set $U \subseteq \R^d$ with $x_0 \in U \cap \X$ and for all $x_1 \in U \cap \X$ we have $\twoVariableFunction(x_1,y_0)<0$. Thus, applying the sign-monotonicity of $y \mapsto \twoVariableFunction(x_0,y)$, we have $\implicitOneVariableFunction(x_0)-\epsilon <y_0\leq \inf\{y \in \Y : \twoVariableFunction(x_1,y)\geq 0\} \leq \sup \Y$. Hence, $\implicitOneVariableFunction(x_1) > \implicitOneVariableFunction(x_0)-\epsilon$, and we have shown that $\implicitOneVariableFunction$ is also lower semi-continuous.
\end{proof}

\bibliography{mybib}

\bibliographystyle{apalike}

\end{document}